\documentclass[16pt]{article}
\usepackage{amsmath,amsthm,amsfonts,amssymb}
\usepackage{verbatim}
\usepackage{latexsym}
\newtheorem{theorem}{Theorem}[section]
\newtheorem{lemma}[theorem]{Lemma}

\newtheorem{definition}[theorem]{Definition}
\newtheorem{corollary}[theorem]{Corollary}

\newtheorem{remark}[theorem]{Remark}

\newcommand\RR{{\Bbb R}}

\newcommand\NN{{\Bbb N}}
\newcommand\ZZ{{\Bbb Z}}

\newcommand\HH{{\Bbb H}}

\begin{document}
\title{Shannon Multiresolution Analysis on the Heisenberg Group}

\author{
 \thanks{Research supported by the German Academic  Exchange Service DAAD, {\em German Academic Exchange Service}. 2005}
Azita Mayeli  \\
}


\maketitle

\begin{abstract}

We   present a notion   of frame multiresolution analysis on the Heisenberg group, abbreviated by  FMRA,
  and study its properties. Using the irreducible representations of this group, we shall define
  a sinc-type function which is our starting point for obtaining the scaling function. Further, we shall 
     give a concrete example of a wavelet FMRA on the Heisenberg group which is  analogous to the Shannon
 MRA on $\RR$.

\vspace{.5cm}

\footnotesize{
\begin{tabular}{lrl}
{\bf  Keywords.} Heisenberg group,   Schr\"odinger representations,   group Fourier transformation,
  &  \multicolumn{2}{l} {\em  
 }\\
Plancherel theorem, band-limited functions, 
  convolution operator,     wavelet frame. &  \multicolumn{2}{l} {\em  
 }\\
    &  \hspace{-.1cm}  {\em  }\\
 \end{tabular}
 \vspace{.3cm}
 
 \begin{tabular}{lrl}
  &&  \hspace{-1cm}{\bf  AMS  Subject Classification (2000)} {22E25, 22E27, 42C40.
  }
\end{tabular}
}
      \end{abstract}
   
          \section{Introduction}
\label{sec:Introduction}

\textit{Multiresolution analysis} (MRA),  is an important
 mathematical tool since it provides a natural framework for understanding
  and constructing discrete  wavelet systems (wavelet frames). 
  The theory of a frame multiresolution analysis, for instance on $\RR$,
  and some of its properties are studied  in  \cite{BeLi93}. \\

  There are  different approaches to construct wavelet frames  on the Heisenberg  group $\HH$ (e.g. disceretization of a continuous wavelet \cite{gm1}).  In the present work, 
   we shall define and present  a  frame multiresolution analysis FMRA on $\HH$,  which will imply the existence 
 of  a normalized 
tight wavelet frame (n.t. frame) on this group. 
   More precisely, in Theorem \ref{prove-of-wavelet} we shall show that:\\

{\em
There exists a band-limited function $\psi\in L^2(\HH)$ and a lattice
$\Gamma$
in  $\HH$ such that
   the discrete wavelet system $\{L_{2^{-j}\gamma}D_{2^{-j}}\psi\}_{j,\gamma}$ forms a n.t. frame
of
$L^2(\HH)$.}\\

Accordingly, any function in $L^2(\HH)$ can be expanded in this
wavelet frame with associated wavelet coefficients. \\

 A standard way to construct a wavelet frame by  the multiresolution analysis technique is 
by starting with a scaling function, i.e., a function which is refinable.
   In contrast to the  standard approach, our starting point here is not a scaling function.
 Rather,  we first construct a sinc-type function on $\HH$ which is band-limited, self-adjoint, 
and has additional properties as in Theorem \ref{sinc-type-function}. The existence of a sinc-type 
function with the desired properties implies the existence of  a scaling and wavelet function on $\HH$. \\

    A different notion of   mutliresolution analysis on stratified Lie groups was obtained  by 
     Lemari{\'e} \cite{Le89}. There, in fact, 
   the left-translations of the scaling function under 
      a discrete set  constitute an unconditional basis for the central scaling space.   
   As an example, he constructed a MRA, which arises 
     from  a  generalized spline-surface space, and obtains a $C^N$ wavelet orthonormal basis of 
  spline wavelets  on these groups.  
   The wavelet orthonormal basis in this example is generated by finitely  many wavelets.\\

      Continuous  wavelets on nilpotent Lie groups have been studied by many authors (e.g.,  see  \cite{Brad07, Fuehr05}
 and references therein).  The existence of a Parseval frame on $\HH$, 
 which is a system of   dilates and left-translations of a single wavelet, is proved   in \cite{Brad08}, 
      The existence of a continuous  wavelet  in  closed subspaces of  $L^2(\HH)$ was studied   in \cite{LiuPeng97}. 
(For the definition of continuous wavelet on $\HH$, see for instance \cite{mayeli06}). The authors in \cite{LiuPeng97}
 do not study the disceretization of the wavelet to obtain  a wavelet frame or wavelet orthonormal basis.\\

 The   contributions  of this work will be as follows: After the introduction and some notation and preliminaries in Section \ref{Pre-Not}, in Section
  \ref{Groups-Fourier-Transformation-of-the-Heisenberg-Group}
 we shall give a brief review of the group Fourier Analysis on the Heisenberg group.   
  The main results of this work are presented 
in Section
\ref{sec:A FrameMultiresolutionAnalysisForHeisenbergGroup}. Here, we shall introduce  the  
  \text{FMRA}  on $\HH$, i.e., the concept of orthonormal basis
 will be replaced by frames. Then we present  a concrete example of
    FMRA on $\HH$, Shannon MRA, and hence we prove  the existence of a scaling and
wavelet
     function for the Heisenberg group.\\
Finally,  we demonstrate the existence of a Shannon
normalized tight frame on $\HH$, i.e., existence
 of a \textit{band-limited}  function on $\HH$
  such that its translations under an appropriate
  lattice in   $\HH$ and its dilations
   with respect to  the integer powers of a  suitable
      automorphism of $\HH$ yields
    a normalized tight frame for $L^2(\HH)$.\\

  Some words about  MRA: 
  There are three
 things in MRA that mainly concern us:   the density of the
 union,  the triviality of the intersection of the nested sequence of
 closed subspaces, and the existence of refinable functions, i.e.,
  functions which have an expansion in their scaling.
  The  triviality  of the intersection  is derived from the
  other conditions of MRA. To obtain
the density of the union, we have to generalize the
concept of the \textit{support} of the Fourier transform.
The new concepts, such as  \textit{band-limited} in $L^2(\HH)$,
  arise
 in this generalization. As to refinability, it depends very
 much on the individual function $\phi$, the so-called
  {\em scaling function}. An example of a  scaling function is presented in this work.\\
Then, we create the  Sha\-n\-non\--\-MRA, as a  concrete example of a FMRA on the Heisenberg
group, for which we prove the
existence of a wavelet function. This wavelet function
is related to a certain lattice of $\HH$. 

  Although the central closed shift-invariant space 
  in our multiresolution analysis is a Paley-Wiener space on the Heisenberg group, we do not study 
  any sampling theorems here.  
For  sampling theory   in the general setting, see  for instance \cite{Feich-Pesen04},
 \cite{FuGr}, and \cite{Pesenson98}  and  the references therein.

Observe that we do not obtain any
smoothness conditions on our wavelets here. A class of Schwartz wavelets in the  general setting, 
i.e.,  stratified Lie groups, has already been constructed and studied  in our earlier work  \cite{gm1}.
Those wavelets did not arise from an MRA, and are only ``nearly tight"  (if sufficiently fine lattices are
used).

  
     \section{Preliminaries and Notations}\label{Pre-Not} 
      We use the abbreviation  ONB  for orthonormal basis and the word  {projection} for
self-adjoint projection operator on a Hilbert space.  We denote the space of Hilbert-Schmidt
operators on $L^2((\RR)$ by  $HS(L^2(\RR))$
  (For the facts we shall use about Hilbert-Schmidt operators and trace-class operators, see \cite{Folland95},
Appendix 2, and \cite{Schatten60}, $\S 2$ and $\S 3$.)\\

In the following, we outline some notation and results concerning direct integrals. For further
information on direct integrals, we refer the reader to \cite{Folland95} $\S7.4$.\\
A family $\{\mathcal{H}_\alpha\}_{\alpha\in A}$ of nonzero separable Hilbert spaces indexed by $A$
will be called a field of Hilbert spaces over $A$.  We assume $A$ is a topological space with a Borel$\sigma$-algebra.
A map $f$ on $A$ such that $f(\alpha)\in \mathcal{H}_\alpha $ for each $\alpha\in A$ will be
called a
  \textit{vector field} on $A$. We denote the inner product and norm on $\mathcal{H}_\alpha $ by
$\langle\;,\;\rangle_\alpha$ and $\parallel.\parallel_\alpha$. A measurable field of Hilbert space 
over $A$ is
a
field of Hilbert spaces
$\{\mathcal{H}_\alpha\}_{\alpha\in A}$ together with a countable family
 $\{e_j\}_1^\infty$ of vector fields with the following properties:
 \begin{itemize}
\item[(a)] the functions $\alpha\mapsto \langle e_j(\alpha), e_k(\alpha)\rangle_\alpha$ are measurable 
for all
$j,k$,
\item[(b)] the linear span of $\{e_j(\alpha)\}_1^\infty $ is dense in $\mathcal{H}_\alpha$, for each
$\alpha$.
\end{itemize}
Given a measurable field of Hilbert spaces $\{\mathcal{H}_\alpha\}_{\alpha\in A}$,  $\{e_j\}$ on
$A$, a vector field $f$ on $A$ will be called {\bf measurable}
 if the function $\alpha\rightarrow \langle f(\alpha), e_j(\alpha)\rangle_\alpha$
is measurable function on $A$, for each $j$.
Finally, we are ready to define direct integrals. Suppose $\{\mathcal{H}_\alpha\}_{\alpha\in A}$ ,
$\{e_j \}_1^\infty $  is a measurable field of Hilbert spaces over $A$, and suppose $\mu$
is a measure on $A$.
  The direct integral of the spaces $\{\mathcal{H}_\alpha\}_{\alpha\in A}$ with respect to $\mu$
is
denoted by
$
\int_A^{\bigoplus}\mathcal{H}_\alpha ~d\mu(\alpha).
$
This is the space of measurable vector fields $f$ on $A$ such that
\begin{align}\notag
\parallel f\parallel^2= \int_A\parallel f(\alpha)\parallel_\alpha^2 d\mu(\alpha)<\infty,
\end{align}
where two vector fields agreeing almost everywhere are identified.
 Then it easily follows that $\int^{\bigoplus}\mathcal{H}_\alpha d\mu(\alpha)$ is a Hilbert space
with the inner product
\begin{align}\notag
\langle f,g\rangle =\int_A \langle f(\alpha),g(\alpha)\rangle_\alpha d\mu(\alpha).
\end{align}
In case of a constant field, that is, $\mathcal{H}_\alpha=\mathcal{H} $ for all $\alpha\in A$,
$\int^{\bigoplus}\mathcal{H}_\alpha d\mu(\alpha)= L^2(A,\mu,\mathcal{H} )$, all the measurable
functions $f: A\rightarrow
\mathcal{H} $ defined on the measure space $(A,\mu)$ with values in  $\mathcal{H}$ such that
\begin{align}\notag
\parallel f\parallel^2= \int_A \parallel f(\alpha)\parallel^2 d\mu(\alpha)<\infty.
\end{align}
 Here
  $\mathcal{H}$  is  considered as a Borel space with
 the
 Borel-$\sigma$-algebra of the norm topology. We will be taking $\mathcal{H}= L^2(\RR)$. \\

 \subsection{Heisenberg Group}
 The Heisenberg group $\HH$ is  a Lie group with underlying manifold $\RR^3$. We denote points in
$\HH$ by $(p,q,t)$ with $p,q,t\in \RR$, and define the group operation by
\begin{align}\label{group-action}
(p_1,q_1,t_1)\ast (p_2,q_2,t_2)= \big(p_1+p_2,q_1+q_2,t_1+t_2+\frac{1}{2}(p_1q_2-q_1p_2)\big).
\end{align}
It is straightforward to verify that this is a group operation, with the origin $0=(0,0,0)$ as the
identity element. Note that the inverse of $(p,q,t)$ is given by $(-p,-q,-t)$.
We can identify both $\HH$ and its Lie algebra $\mathfrak{h}$ with $\RR^3$, with group operation
given by  (\ref{group-action}) and Lie bracket given by
\begin{align}\label{bracket}
\left[(p_1,q_1,t_1),(p_2,q_2,t_2)  \right]:= (0,0,p_1q_2-q_1p_2).
\end{align}
The Haar measure on the Heisenberg group $\HH= \RR^3$ is the usual Lebesgue measure.
More precisely, the Lie algebra $\mathfrak{h}$ of the Heisenberg group $\HH$ has a basis
$\{X,Y,T\}$, which we may think of as left invariant differential operators on $\HH$; where   $[X,Y]=T$  and  all other  brackets are zero,  and where the exponential function $\exp:\;\mathfrak{h}\rightarrow \HH$ is the  identity, i.e.,
\begin{align}\notag
\exp\;(pX+qY+tT)= (p,q,t).
\end{align}
We define the  action of  $\mathfrak{h}$ on the space $C^\infty(\HH)$ via
  left invariant differential operators.\\
  For two functios $f$ and $g$, each in $L^1(\HH)$ or $L^2(\HH)$, the {convolution}
of $f$ and $g$ is the function $f\ast g$ defined by
\begin{align}\label{convolution}
f\ast g(\omega)=\int_\HH f(\nu)g(\nu^{-1}\omega)d\nu.
\end{align} 
We note that, for any pair $f,g\in L^2(\HH)$, one has $f\ast \tilde{g}\in C_b(\HH)$, where
$\tilde{g}(\omega)=
\overline{g(\omega^{-1})}$.
For more details  about convolution of functions see for example \cite{Folland95} Proposition
(2.39).
 \begin{definition}
$f\in L^2(\HH)$ is called \text{selfadjoint convolution idempotent} if $f=\tilde f= f\ast f.$
\end{definition}
The selfadjoint convolution idempotents  and their support properties are studied in detail  in 
  \cite{Fuehr05}\S2.5.\\

  For properties of selfadjoint convolution idempotents  in $L^2$ we refer the reader to
\cite{Fuehr05},\;\S 2.5.\\

Our definition of a  continuous, or a discrete, wavelet  on $\HH$, involves the 
  one-parameter dilation group of $\HH$, i.e., $H=(0,\infty)$, where any $a>0$ defines an
automorphism of $\HH$,
  by
\begin{align}\label{oneparameterauto}
a(p,q,t)= (ap,aq,a^2t)\hspace{.5in}\forall ~(p,q,t)\in\HH.
\end{align}
(In the construction of a discrete wavelet, one takes  a discrete version of the one-parameter group.
Usually a dyadic discretization is   considered.)
Adapting the notation of dilation and
 translation  operators on $L^2(\RR)$, for each $a>0$, 
we define $D_a$ to be the unitary operator on $L^2(\HH)$
given by
\begin{align}\notag
D_af(p,q,t)= a^{2}f(a(p,q,t))= a^{2}f(ap,aq,a^2t)\quad \forall f\in L^2(\HH),
\end{align}
 and for any $\upsilon\in \HH$, the left translation operator, $L_\omega$ is given by
\begin{align}\notag
L_\omega f(\upsilon)=  f( \omega^{-1}\upsilon)\quad\forall \upsilon\in \HH.
\end{align}
Using the dilation and translation operators, we can now define the \textit{quasiregular representation } $\pi$ of
the
    semidirect product   $G:=\HH\rtimes (0,\infty)$; it acts on $L^2(\HH)$  by
 \begin{align}\notag 
 (\pi(\omega,a)f)(\upsilon):=L_\omega D_af(\upsilon)= a^{-2}  f(a^{-1}(\omega^{-1}\upsilon)),
 \end{align}
for any $f\in L^2(\HH)$ and $(\omega,a)\in G$  and for all $\upsilon\in\HH$.\\
 
 \subsection{Frames}
We conclude this section with the definition of  \textit{frames} and
      some related  notions, which  will be used
    in the context of  FMRA  in this work. 
 The concept of  frames is a generalization of orthonormal bases, defined as follow:
 \begin{definition}\label{definition-of-frame}
A countable subset $\{e_n\}_{n\in I}$ of a Hilbert space $\mathcal{H}$ is said to be a
\textbf{frame} of $\mathcal{H}$ if there exist two  numbers $0<a\leq b$ so that, for any
   $f\in\mathcal{H}$,
   \begin{align}\notag
   a\parallel f\parallel^2 \leq \sum_{n\in I} \mid \langle f,e_n\rangle \mid^2 \leq b\parallel 
f\parallel^2.
   \end{align}
 The positive numbers $\textbf{a}$ and $\textbf{b}$ are called
  \textbf{frame bounds}.
 Note that the frame bounds are not unique. The \textit{optimal lower frame bound} is the supremum
over all lower  frame bounds, and the  optimal  upper  frame  bound    is the infimum over all
upper frame bounds. The optimal frame bounds are actually  frame bounds.
  The frame is called a \textbf{tight frame}  when one can take $a=b$ and a
  \textbf{normalized tight frame} when one can take $a=b=1$.
 \end{definition}
Frames were introduced for the first time in
\cite{DuffinSchaeffer52}. See also \cite{Christensen03} and \cite{Groechening01} for more about frame theory.


 \section{ Fourier Analysis on the Heisenberg Group}\label{Groups-Fourier-Transformation-of-the-Heisenberg-Group}

 This section contains a brief review of Fourier analysis on the Heisenberg group  $\HH$.
      In order to study the Fourier analysis  on this  group, one  has  to study  the irreducible representations  of this group.
    The Heisenberg group  is the best known example of a non-commutative nilpotent Lie group. The
representation theory of $\HH$ is simple and well understood. Using the fundamental theorem, due
to
\textit{Stone and von Neumann}, we can give a complete classification of all the irreducible
unitary
representation of $\HH$.\\
It is known that for  the Heisenberg group  there are two families of irreducible unitary representations, at least
 up to unitary equivalence. One family, giving all infinite-dimensional irreducible unitary
 representations, is parametrized by nonzero real numbers $\lambda$; the other family, giving all
 one-dimensional representations, is parametrized by $(b,\beta)\in \RR\times \RR$.
  We will see below that the one-dimensional representations give no contribution to the Plancherel
formula and Fourier inversion transform, i.e.  they form a set of representations that has zero
Plancherel measure. 
 Hence we will focus on the {Schr\"odinger} representation, defined
next.   For more about the representations of the Heisenberg group and the Plancherel theorem, we refer the 
interested reader to  \cite{Geller77}.\\
   
       The infinite-dimensional irreducible unitary representations of the Heisenberg group may
be realized on $L^2(\RR)$; there they are the called the
  \textit{Schr\"odinger} representations.  These are defined as follows.
  For each $\lambda\in \RR^\ast(=\RR \backslash \left\{ 0\right\})$ and  for any $(p,q,t)\in \HH$, the
operator $\rho_\lambda(p,q,t)$ acts on $L^2(\RR)$ by
  \begin{align}\label{schroedinger}
 \rho_\lambda(p,q,t)\phi(x)= e^{i\lambda t}e^{i\lambda(px+\frac{1}{2}(pq))}\phi(x+q)
\end{align}
where $\phi\in L^2(\RR)$. It is easy to see that $\rho_\lambda(p,q,t)$  is a unitary operator
satisfying the homomorphism property:
\begin{align}\notag
\rho_\lambda\big((p_1,q_1,t_1)(p_2,q_2,t_2)\big)= \rho_\lambda(p_1,q_1,t_1)
\rho_\lambda(p_2,q_2,t_2).
\end{align}
Thus each $\rho_\lambda$ is a strongly continuous unitary representation of $\HH$.
A theorem of Stone and von Neumann  (\cite{Folland95}) says that up to unitary equivalence these are all the
infinite-dimensional irreducible unitary representations of $\HH$.\\
 Recall that the
dilation  operator
given by  $a>0$
is defined on $\HH$ as follow:
\begin{align}\notag
a: \;(p, q, t)\rightarrow a(p,q,t)=(ap,aq, a^2t)\quad \forall \;(p,q,t)\in \HH.
\end{align}
One then easily calculates that
 \begin{align}\label{equivalent-representations}
  \rho_\lambda(a^{-1}(p,q,t))=
  D_{a^{-1}} \rho_{a^{-2}\lambda}(p,q,t)D_a\quad \forall (p,q,t)\in \HH
  \end{align}
where  $D_{a^{-1}}=D_a^\ast$.\\

 \subsection{Fourier Transform on the Heisenberg Group}
 Here  we present a brief introduction to  the group Fourier transform for functions on $\HH$, (see \cite{Geller77}), and introduce the
inversion  and Plancherel theorems for the Fourier transform. \\

If $f\in L^1(\HH)$, we define the Fourier transform of $f$ to be the measurable field of operators
over $\widehat \HH$ given by the weak operator integrals, as follows:
 \begin{align}\label{fourier-transform}
\widehat f( \lambda)=\int_\HH f(\omega)\rho_\lambda(\omega)~d\omega.
\end{align}
For simplicity, we write here $\widehat f(\lambda)$ instead of $\widehat f(\rho_\lambda)$. Note that the
Fourier transform $\widehat f( \lambda)$ is an
  operator-valued function, which for any $\phi, \psi\in L^2(\RR)$ satisfies
 \begin{align}\notag
 \langle\widehat f(\lambda)\phi,\psi\rangle= \int_\HH f(p,q,t) \langle 
\rho_\lambda(p,q,t)\phi,\psi\rangle
 ~dpdqdt,
\end{align}
by definition of the weak operator integral.
The operator $\widehat f(\lambda)$ is bounded on $L^2(\RR)$ with the operator norm satisfying
$\parallel \widehat f(\lambda)\parallel \leq \parallel f\parallel_1.$ 
If $f\in L^1\cap L^2(\HH)$, $\widehat f(\lambda)$ is actually a Hilbert-Schmidt operator and from  the Plancherel theorem, the
Fourier transform can be extended to a unitary map from  $  L^2(\HH)$ onto   $ L^2\big(\RR^\ast,   d\mu(\lambda), L^2(\RR)\otimes L^2(\RR) \big),$ 
  the space of functions on $\RR^\ast$ taking values in
$L^2(\RR)\otimes L^2(\RR)$ which are square
integrable with respect to  the Plancherel measure $ d\mu(\lambda)= (2\pi)^{-2}\left|\lambda\right| d\lambda$.   (We say  a function 
$g$ defined on $\RR^\ast$  is square integrable, when it is a measurable vector field on $\RR^\ast$  and 
$\int_{\RR^\ast} \parallel g(\lambda)\parallel_{H.S}^2 d\mu(\lambda)< \infty.$) 
The proof of the Plancherel theorem  for the Heisenberg group may be found in \cite{Geller77}.
For more general groups, see for example \cite{Folland95}.
  
 A simple computation shows that the basic properties of the Fourier transform remain valid for
$f,g\in
\left(L^1\cap L^2\right)(\HH)$: More precisely, 
\begin{itemize}
\item[(i)] $\widehat{(af+bg)}(\lambda)=a\widehat f(\lambda)+b\widehat g(\lambda),$
\item [(ii)]$\widehat{(f\ast g)}(\lambda)=\widehat f(\lambda)\widehat g(\lambda),$
\item [(iii)]$\widehat{(L_\omega f)}(\lambda)=\rho_\lambda(\omega)\widehat f(\lambda),$ for $\omega\in\HH$
 \item [(iv)]$\widehat{( \tilde f)}(\lambda)=\widehat f(\lambda)^\ast.$  
 ( The superscript   $\ast$ denotes adjoint.)
\end{itemize}
We conclude this section with a computation of the Fourier transform of $f(a\cdot)$ for any $f\in L^2(\HH)$. 
 \begin{lemma} For any $f\in L^2(\HH)$ is
 \begin{align}\notag
 \widehat{f(a\cdot)}(\lambda)
 =a^{-4} D_{a^{-1}}\widehat {f(a^{-2}\lambda)}D_a.
 \end{align}
 \end{lemma}
 \begin{proof}
  From the definition of the Fourier transform (\ref{fourier-transform}), for $\lambda\not=0$  we have
 \begin{align}\notag
 \widehat{f(a\cdot)}(\lambda)&= \int_\lambda f(a(p,q,t) )
\rho_\lambda(p,q,t)~dpdqdt\\\notag
 &= \int_\lambda f(ap,aq,a^2t) \rho_\lambda(p,q,t)~dpdqdt\\\notag
 &=a^{-4}\int_\lambda f(p,q,t) \rho_\lambda (a^{-1}p,a^{-1}q,a^{-2}t)~ dpdqdt\\ \notag
 &=a^{-4}\int_\lambda f(p,q,t) \left(\rho_\lambda(a^{-1}(p,q,t)\right) ~dpdqdt,
 \end{align}
Now inserting   (\ref{equivalent-representations}),   we derive the following relation:
 \begin{align}\label{dilation-of-function}
 \widehat{f(a\cdot)}(\lambda)&= a^{-4} D_{a^{-1}}\big(\int_\lambda f(p,q,t)
\rho_{a^{-2}\lambda}(p,q,t)~dpdqdt\big)D_a \\\notag
 &=a^{-4} D_{a^{-1}}\widehat{f(a^{-2}\lambda)}D_a,
 \end{align}
 as desired.
 \end{proof}

\subsection{Wavelet Frames}
\label{sub-sec:WaveletAnalysis}

  For our purpose, in this work 
   we will consider
   the  wavelet frames
   which are  produced  from one function, as the generator of the wavelet frame,   using a countable family of  dilation and left
translation operators. The generator function is  usually called a
   \textit{``discrete wavelet"}.\\
   
   Below we will  give a  concrete  example of wavelet frames
 with respect to a very special lattice as the discrete  translation set.
    Suppose $\Gamma$ is  a lattice in $\HH$ and $a>0$ refers to the automorphism $ a
:\omega\rightarrow a.\omega$ of
$\HH$. Suppose also $ \mathcal{H}$ is a subspace of $L^2(\HH)$ and $\psi$ is any function in 
$\mathcal{H}$. Then the
discrete   system $\{L_{a^{-j}\gamma}D_{a^{-j}} \psi\}_{j\in\ZZ, \gamma\in\Gamma}$ 
(assumed to be contained in $ \mathcal{H}$)
is called the
         \textit{discrete wavelet system}  generated by $\psi$, where $\{D_{a^{-j}}\}_{j\in\ZZ}$  is the class of  discrete 
unitary dilation operators with respect to the positive number $a$  obtained by $a^j:\omega\rightarrow a^j.\omega$, and $\{L_\gamma\}_{\gamma\in \Gamma}$
 is the class of  left translation operators   with
regard to  the lattice $\Gamma$. The discrete wavelet system  $\{ L_{a^{-j}\gamma}D_{a^{-j}} \psi\}_{j\in\ZZ,
\gamma\in\Gamma}$     is called a \textit{(tight, normalized tight) wavelet frame} of $\mathcal{H}$ if it forms a
(tight, normalized tight) frame for $\mathcal{H}$.
 More precisely, it is a frame if there exist positive numbers $0<A<B<\infty$ such that for any $f \in \mathcal{H}$ we have 
 \begin{align}\notag
 A\| f\|_2^2 \leq \sum_{j\in\ZZ, \gamma\in\Gamma} \mid \langle f, \psi_{j, \gamma}\rangle \mid^2 \leq B \|f\|_2^2,
 \end{align}
 where $\psi_{j, \gamma}:= L_{a^{-j}\gamma}D_{a^{-j}} \psi$. \\

  Now, we are ready to present our main results concerning the multiresolution analysis.

 \section{Frame Multiresolution Analysis   for $L^2(\HH)$}
\label{sec:A FrameMultiresolutionAnalysisForHeisenbergGroup}

 Analogous to the situation on $\RR$, discrete wavelets in $L^2(\HH)$ are functions $\psi$ with the property that  their
appropriate  translates and  dilates defined with respect to the Lie structure of the Heisenberg
group
 can be used to  approximate  any $L^2$-function on $\HH$. But
here the special concept of multiresolution analysis needs to be appropriately adapted.\\

In this work we shall  adapt the definition of MRA for $L^2(\RR)$ to
one for $L^2(\HH)$, replacing  the concept of orthonormal basis by  frames and calling it {\em FMRA}.
Since the triviality  of the intersection    is a direct consequence of the other conditions of
the
definition of an MRA, we prove this property immediately after we give the definition of an
MRA. \\
We begin by properly interpreting  the concept of  MRA of $L^2(\RR)$.   The shift-invariance of $V_0$,
the central subspace in the definition of a MRA, 
 can be interpreted
as an
invariance property with respect to the action of the discrete lattice subgroup $\ZZ$ of $\RR$.
The scaling operator $\alpha$ can be viewed as the action of some group automorphism of $\RR$, with the
property $\alpha\ZZ\subset \ZZ$. \\
With this in mind, it is not difficult to conjecture the correct generalization of MRA to the
Heisenberg group :
\begin{itemize}
\item First,
 a discrete subgroup $\Gamma$ of $\HH$ will play the same role in $\HH$ as $\ZZ$ in $\RR$.
 To say $\Gamma$ is discrete means that
the topology on $\Gamma$ induced from $\HH$ is the discrete topology.
 \item Since our setting is   a non-abelian group, there are two kinds of translations:
left translation $L:=L_\HH$ and right translation $R:=R_\HH$.  We choose left translation here. 
   \end{itemize}
   
   As our starting point, we need the following definition:
 \begin{definition}
 Suppose $\Omega$  is a subset of $\HH$ and $\mathcal{H}$ is a subspace of $L^2(\HH)$. We say
$\mathcal{H}$ is left
  shift-invariant under $\Omega$, if for any $\omega\in \Omega$ we have $L_\omega
\mathcal{H}\subseteq \mathcal{H}$.
    \end{definition}
 After this preparation, we can give a definition of  {FMRA}  for $L^2(\HH)$ related
to an automorphism  of $\HH$ given by $a>0$ (see (\ref{oneparameterauto})) and a lattice $\Gamma$
in $\HH$.
 \begin{definition}\label{def:MRA-H} 
    We say that a sequence of closed subspaces $\{V_j\}_{j\in \ZZ}$ of $L^2(\HH)$ forms a
\textit{FMRA }  of $L^2(\HH)$, associated to an automorphism $a\in Aut(\HH)$
and a lattice $\Gamma$ in $\HH$,  if the following conditions are satisfied:
\begin{enumerate}
    \item  $  V_j\subseteq V_{j+1}\quad \forall j \in \ZZ$,
        \item $\overline{\bigcup V_j}= L^2(\HH)$,
    \item $\bigcap V_j= \{0\}$,
    \item $f\in V_j \Leftrightarrow f(a\cdot) \in V_{j+1} $,
        \item $V_0$ is left shift-invariant under $\Gamma$, and consequently   $V_j$ is left
         shift-invariant  under $ a^{-j}\Gamma$, and,
        \item there exist a function $\phi\in V_0$, called the
         \textit{scaling function}, or generator of the  FMRA, such that the set 
$L_\Gamma(\phi)$
constitutes a normalized tight frame for $V_0$.
\end{enumerate}
 \end{definition}

  \begin{remark}
\item [(a)]  Observe that property 4 in Definition  \ref{def:MRA-H}
implies that
\begin{align}\label{equ}
f\in V_j  \Leftrightarrow f(a^{-j}\cdot)\in V_0.
\end{align}
It follows that an  MRA  is essentially completely determined
by the closed subspace $V_0$. But from property 6, $V_0$ is the closure of
the linear span of the $\Gamma$-translations of the scaling function
$\phi$. Thus the starting point of the construction of  MRA is
the existence of the scaling function
$\phi$. Therefore, it is especially important  to give
some conditions under which an initial function $\phi$ generates an MRA.
\item [(b)]  Equation (\ref{equ}) implies
that if $f\in V_j$, then $f(\gamma^{-1}(a^{j}\omega))\in V_0$ for all $\gamma\in \Gamma$.
Finally  property  6  in Definition \ref{def:MRA-H} and  equation
 (\ref{equ}) imply that the system  $\{L_{a^{-j}\gamma}D_{a^{-j}}\phi\}_{\gamma\in \Gamma}$
is a  normalized tight frame $V_j$  for all $j\in \ZZ$, where $\forall \gamma\in \Gamma,\forall \omega\in \HH$,
$L_{a^{-j}\gamma}D_{a^{-j}}\phi(\omega) =a^{j/2}\phi(\gamma^{-1}(a^{j}\omega))$ 
 \item [(c)]
  Here, for the scaling function we do not impose any regularity or decay condition on $\phi$. In our
case to make the argument simple and general, we require only that $\phi\in L^2(\HH)$.\\
     \item [(d)]  In analogy with $L^2(\RR)$, we say $V_0$ is
     refinable if $D_{a^{-1}}(V_0)\subseteq V_0$. Thus condition $1$ in Definition
\ref{def:MRA-H} is equivalent to saying that $V_0$ is refinable. Thus, the basic question
    concerning a FMRA is whether the scaling function exists. We  shall see in Theorem \ref{existence_of_scaling_function} that such
scaling functions do exist. We will enter into details   later  for  a very  special case. \\
 \item [(e)]  To have a sequence of nested closed subspaces, we must find a refinable function like
$\phi$ in $V_0$.
    It is already known by Boor, DeVore and Ron in \cite{BoorDevoreRon93} for the real case that 
the
refinability of $\phi$ is not enough to generate  an $MRA$. Hence we need other requirements. We
will consider this in detail later.
 \item [(f)] The basic property of multiresolution analysis is that whenever
   a collection of closed subspaces satisfies
properties  1-6 in Definition \ref{def:MRA-H}, then there exists a
 basis $\{ L_{a^{-j}\gamma}D_{a^{-j}}\psi ;\;j\in \ZZ, \gamma \in \Gamma\}$ of $L^2(\HH)$, such that for all $f\in
L^2(\HH)$
 \begin{align}\notag
 P_{j+1}f= P_{j}f+\sum_{\gamma \in \Gamma} \langle f,L_{a^{-j}\gamma}D_{a^{-j}} \psi\rangle ~L_{2^{-j}\gamma}D_{2^{-j}} \psi ,
 \end{align}
 where $P_j$ is the orthogonal projection  of $L^2(\HH)$ onto $V_j$.
  \end{remark}
 
The  approach via the  solution of the scaling equation, with methods of Lawton \cite{Lawton02},
leads to   difficult analytical problems. Therefore  we follow a new  approach, which is based on
the point of view ofHermi
 \textit{Shannon multiresolution analysis}. This  will  allow us to derive the   existence of a
Shannon  wavelet   in $L^2(\HH)$.\\
In the setting of the real line, 
  the canonical construction of wavelet bases starts with a
multiresolution analysis $\{V_j\}_{j}$. In $L^2(\RR)$ one proves the existence of a wavelet
$\psi\in W_0$, such that $\{L_k
\psi, \; k\in\ZZ\}$ is an orthonormal basis for
$W_0$. ($L_k$ is the translation operator.)\\
  Consequently the set $\{L_{2^{-j}k }D_{2^j}\psi\}_{k\in \ZZ}$, the set of dyadic dilations and translations of $\psi$, 
 constitutes  an orthonormal basis for $W_j$. (Dyadic refers here to the dilation $D_2$.)
By the orthogonal decomposition $L^2(\RR)=\underset{j\in \ZZ}{\bigoplus}  W_j$, the wavelet system
$\{ L_{2^{-j}k}D_{2^{j}}\psi\}_{j, k\in \ZZ}$ is an orthonormal basis for $L^2(\RR)$. \\

In our setting, we shall construct  on $\HH$ a Shannon-MRA as an example of a FMRA.
In contrast of the case of $\RR$, the construction of the scaling function is not our starting point
for
obtaining a FMRA, but rather, first we intend to construct a special function which implies the
existence of the scaling function in some closed subspace of $L^2(\HH)$.
 Furthermore, for the construction, we shall consider the automorphism $a=2$ of $\HH$ which  is
given by:
\begin{align}\notag
 a(p,q,t)= (2p,2q,2^2t)\hspace{.3in} \forall~ (p,q,t)\in \HH.
  \end{align}
  
  \subsection{An Example: Shannon MRA for $L^2(\HH)$}

 As remarked before, we shall  construct a generator  function in some closed and shift-invariant subspace
of $L^2(\HH)$, such that  its translations and dilations yields a
 normalized tight frame of  $L^2(\HH)$. 
    For this reason, first  we choose the     dilation  operator $D_a= D_2$
    and try to associate a space $V_0$ which has similar properties as the Paley-Wiener space on
$\RR$. With this aim in mind, 
      we start with the
definition of a band-limited function on $\HH$:
      \begin{definition}\label{strongly-supported}  Suppose $\mathcal{I}$ is  some bounded subset
of
 $\RR^\ast$ and
   $S$ is a function in $L^2(\HH)$. We say $S$ is $\mathcal{I}$-band-limited  if $\widehat
S(\lambda)=0$ for all $\lambda\not\in \mathcal{I}$.
   \end{definition}
   
  In the next theorem   we shall  construct 
    sinc-type function on the Heisenberg group which is our  starting  point for obtaining  the scaling function:
    
   \begin{theorem}\label{sinc-type-function}
   Let $d$ be any positive integer.
   There exists a selfadjoint convolution idempotent function $S$ in $L^2(\HH)$ which is
$\mathcal{I}$-band-limited for  $\mathcal{I}= \left[ -\frac{\pi}{2d},\frac{\pi}{2d} \right]\backslash\{0\}$.
  Define $S_j= 2^{4j}S(2^j\cdot)$ for $j\in\ZZ$. Then   $S_j$ is $\mathcal{I}_j$-band-limited for
    $\mathcal{I}_j=  \left[ -\frac{2^{2j}\pi}{2d},\frac{2^{2j}\pi}{2d} \right]\backslash\{0\}$ 
 and the following consequences hold:
   \begin{itemize}
   \item [$(a)$] $S\ast S_j=S\;\;\forall j> 0\;and\; S_j\ast S= S_j\;\;  \forall j<0$,
  \item [$(b)$] $f\ast S_j\rightarrow 0\;\text{ in} \;L^2\text{-norm \; as}\; j\rightarrow -\infty
\;\forall f\in L^2(\HH)$,
   \item[$(c)$] $f\ast  S_j\rightarrow f \;\text{ in} \;L^2\text{-norm  \; as}\; j\rightarrow \infty
\;\forall f\in L^2(\HH)$, and,
\item [$(d)$] $S_j=\widetilde{S_j}= S_j\ast S_j$.
     \end{itemize}
    \end{theorem}
\begin{proof}
Take  $\mathcal{I}_0:= \mathcal{I}$.  We intend to show that there exists a function $S$ which
is $\mathcal{I}_0$-band-limited and
satisfies the assertion of our theorem.
   We start from the Fourier transform side, i.e, by constructing Hilbert-Schmidt operators  $\hat
S(\lambda)$
    associated to   $\lambda\in \RR^\ast$. For this  purpose we choose an
    orthonormal basis $\{e_i\}_{i\in \NN_0}$ in $L^2( \RR)$.  For any $ \lambda\not=0$  define
   $e_i^\lambda= D_{\left|\lambda\right|^{-1/2}}e_i$. Observe that for any $\lambda$,
    $\{e_i^\lambda\}_i$ is an ONB of $L^2(\RR)$ since the dilation operators
$D_{\left|\lambda\right|^{-1/2}}$
    are
      unitary. Therefore $\{\{e_i^\lambda\}_i\}_\lambda$ is a measurable family of orthonormal
bases in $L^2(\RR)$.
     \\

  Let $\lambda\not=0$ be such that  $\lambda\in \mathcal{I}_0$. For $I_0= \cup_k I_0^k$ where 
  $I_0^k= [ -\frac{\pi}{2^{2k+1}d},  -\frac{\pi}{2^{2k+3}d}) \cup (\frac{\pi}{2^{2k+3}d}, \frac{\pi}{2^{2k+1}d}]$, 
  define the operator 
  $\widehat S(\lambda)$ as follows:
  \begin{align}\notag
\widehat{S}(\lambda)=
\begin{cases}
\sum_{i=0}^{2^{2k}}\big( e_i^{\frac{\lambda}{2\pi}}\otimes e_i^{\frac{\lambda}{2\pi}}\big)
&\text{if }
\lambda \in I_0^k,
 \;\;\text{for some }
k\in \NN_0,\\
\notag
 0 &  \text{otherwise}.
\end{cases}
\end{align}
Therefore for any  $ \lambda \in I_0^k$, 
the operator $\widehat S(\lambda)$
 is a projection operator  on the first $2^{2k}+1$  elements of the  orthonormal basis
$\{e_i^{\frac{\lambda}{2\pi}}\}_{i\in \NN_0}$, where $e_i^{\frac{\lambda}{2\pi}}= D_{\left|
\frac{\lambda}{2\pi}
\right|^{-1/2}}e_i$.
 The definition of $\widehat S$ entails  the following consequences:
\begin{itemize}
  \item [(i)]\;  For  $k\geq 0$ and  
$\lambda \in I_0^k$ is 
 $\parallel \widehat{S}(\lambda) \parallel_{H.S}^2=
2^{2k}+1  $
   \item[(ii)]   $   \int_{\mid \lambda\mid \leq \frac{\pi}{2d}} \parallel
\widehat{S}(\lambda)\parallel_{H.S}^2
  d\mu(\lambda)
 =   \underset{k=0}{\sum} \int_{ I_0^k}
(2^{2k}+1) d\mu(\lambda)
 <\infty,\\
  where\; d\mu(\lambda)= (2\pi)^{-2}\left|\lambda\right|d\lambda,\;  and $
    \item [(iii)]\;$\widehat{S}(\lambda)=\widehat S(\lambda)^\ast= \widehat S(\lambda) \circ
\widehat S(\lambda), \; \forall \lambda\neq 0.$
\end{itemize}
Observe that  (ii) implies    that the vector field $\{\widehat
S(\lambda)\}_\lambda$ on $\RR^\ast $ is contained in $ \int_{\RR^\ast}^\oplus
L^2(\RR)\otimes L^2(\RR)d\mu(\lambda) $  and hence, by the surjectivity part of the Plancherel theorem,
$\widehat S$  has a preimage  $S$ in $L^2(\RR)$ with Fourier transform  $\widehat S$,
given as above.
Property (iii) shows  that $S$ is a self-adjoint   convolution, idempotent by  the
convolution theorem.\\

   Suppose $j\in \ZZ$ and  $S_j:= 2^{4j}S(2^j.)$.
  Using the equivalence of the representations $\rho_\lambda$ and
  $\rho_{2^{-2j}\lambda}$, the relation
 (\ref{dilation-of-function}) and the fact that $D_{2^j}^\ast= D_{2^{-j}}$  we obtain
  \begin{align}\label{S-j}
 \widehat{S_j}(\lambda)= D_{2^{-j}}\widehat S(2^{-2j}\lambda)D_{2^j}.
  \end{align}
(\ref{S-j}) implies
  $\widehat{S_j}(\lambda)=0$ for any
  $\left| \lambda\right| > \frac{2^{2j}\pi}{ 2d}$, and hence 
  the function $S_j$ is $\mathcal{I}_j$-band-limited, where
$\mathcal{I}_j=[-\frac{2^{2j}\pi}{2d}, 0)\bigcup (0, \frac{2^{2j}\pi}{2d}]$.
As a consequence of (iii), the relation (\ref{S-j}) shows that   $S_j$ is a
self-adjoint
    and convolution idempotent, which proves $(d)$. \\
To prove $(a)$, suppose $j>0$ and  $\lambda\in  \mathcal{I}_j$. Then
  $ 2^{-2j}\lambda \in \mathcal{I}_0$.  Hence there exists a  non-negative integer $k_j$ such that
  
$$\frac{\pi}{ 2^{(2k_j+3)}d}<  \left| 2^{-2j}\lambda\right| \leq \frac{\pi}{ 2^{(k_j+1)}d},$$

or equivalently
$ \lambda \in I_0^{k_j}.$

For the case $k_j<j$, observe that $\widehat S(\lambda)=0$.  For the case
$k_j\geq j$,
 from the definition of $\widehat{S}$   we have the following:
  \begin{align}\notag
  \widehat{S}(\lambda)&= \sum_{i=0}^{2^{2(k_j-j)}}
 e_i^{\frac{\lambda}{2\pi}  }\otimes e_i^{\frac{\lambda}{2\pi}}\quad \text{and} \\
  \label{for-S-j}
\widehat S(2^{-2j}\lambda)&= \sum_{i=0}^{2^{2k_j}}
 e_i^{ \frac{ \lambda}{2^{2j+1}\pi} }\otimes  e_i^{ \frac{ \lambda}{2^{2j+1}\pi} }.
  \end{align}
  
Recall that, from the definition of the family of  orthonormal bases $\{e_i^\lambda\}_i$,
 $e_i^{ \frac{ \lambda}{2^{2j+1}\pi} }$
can be read as below:
\begin{align}\label{dilation-of-ONB}
e_i^{ \frac{ \lambda}{2^{2j+1}\pi} }= D_{\left|  \frac{2^{-2j}\lambda}{2\pi}\right|^{-1/2}}e_i=
  D_{2^j}\left(D_{\left| \frac{\lambda}{2\pi}\right|^{-1/2}}e_i\right)= D_{2^j
}e_i^{\frac{\lambda}{2\pi}}.
  \end{align}
Plugging $(\ref{dilation-of-ONB})$  into $(\ref{for-S-j})$, we get
  \begin{align}\notag
\widehat S(2^{-2j}\lambda)=& \sum_{i=0}^{2^{2k_j}} \left(D_{2^j }
 e_i^{ \frac{\lambda}{2\pi}  }\right)\otimes \left(D_{2^j }e_i^{ \frac{\lambda}{2\pi}}\right),
  \end{align}
and hence
  \begin{align}\label{DSj}
  \widehat{S_j}(\lambda)=D_{2^{-j}}\widehat S(2^{-2j}\lambda)D_{2^{j}}=
  \sum_{i=0}^{2^{2k_j}}
 e_i^{ \frac{\lambda}{2\pi}}
  \otimes e_i^{ \frac{\lambda}{2\pi}}.
   \end{align}
  Observe that for any $\lambda\in \mathcal{I}_j$,  the operator $\widehat S(\lambda)$
  is a projection on the first $2^{2(k_j-j)}+1$
elements of the orthonormal basis $\{ e_i^{ \frac{\lambda}{2\pi}}\}$ for some
suitable $k_j\geq j$,
whereas
$\widehat{ S_j}(\lambda)$ is a projection on the  first $2^{2k_j}+1$ elements
of  the same orthonormal basis. Hence we get
\begin{align}\label{for-negative-j}
\widehat S(\lambda)\circ \widehat{ S_j}(\lambda) = \widehat{ S_j}(\lambda)
 \circ\widehat S(\lambda)
 =\sum_{i=0}^{2^{2(k_j-j)}} e_i^{ \frac{\lambda}{2\pi}} \otimes e_i^{ \frac{\lambda}{2\pi}}
 = \widehat S(\lambda),
\end{align}
which is a projection on the first $2^{2(k_j-j)}+1$ elements
of  the  orthonormal basis $\{ e_i^{ \frac{\lambda}{2\pi}}\}$. For fixed  $j>0$  since
the relation (\ref{for-negative-j}) holds for any
 $\lambda\in \mathcal{I}_j $, so by applying
the convolution and the Plancherel theorem respectively we obtain $S\ast S_j=S$, which proves the first hypothesis of (a). \\
Likewise for
  $j<0$, suppose $\lambda\in \mathcal{I}_j$. Then for some $k_j\in \NN_0$,
  $\frac{\pi}{ 2^{2k_j+3}d}<  \left| 2^{-2j}\lambda\right| \leq \frac{\pi}{ 2^{2k_j+1}d}$.
   Analogous to the previous case, the operator $\widehat S(\lambda)$
    is a projection on the first $2^{2(k_j-j)}+1$
elements of the orthonormal basis $\{ e_i^{ \frac{\lambda}{2\pi}}\}$ and
$\widehat{ S_j}(\lambda)$ is a projection on the  first $2^{2k_j}+1$ elements
of  the same orthonormal basis. Thus
\begin{align}\label{for-positive-j}
\widehat S(\lambda)\circ \widehat{ S_j}(\lambda) = \widehat{ S_j}(\lambda)
 \circ\widehat S(\lambda)
 =\sum_{i=0}^{2^{2k_j}} e_i^{ \frac{\lambda}{2\pi}} \otimes e_i^{ \frac{\lambda}{2\pi}}
 = \widehat{S_j}(\lambda).
\end{align}
  Once again,
 applying the convolution and Plancherel theorems in the relation (\ref{for-positive-j}) yields
$
S\ast S_j= S_j$, and hence $(a)$ is completely proved.
\\
 To prove $(b)$,
   suppose $j\in \ZZ$ and $f\in L^2(\HH)$.  Then $f\ast S_j\in L^2(\HH)$ by the structure and properties of the
   function $S$. Before we start to give a proof  for this part, observe
    that, for any $\lambda\not=0$, since each $\widehat{S}(\lambda)$ is a projection,
    the operator
   $\widehat{S}(\lambda)$  is bounded and has operator norm  less than or equal to $1$. Hence for any
$j\in\ZZ$  and
    $\lambda\not=0$
  we have
$$\parallel \widehat{S_j}(\lambda)\parallel_\infty=
 \parallel D_{2^{-j}}\widehat S(2^{-2j}\lambda)D_{2^j}\parallel_\infty\leq 1.$$
 Using the inequality  and
   applying the
       Plancherel and convolution theorems respectively
      we get the followings:
      \begin{align}\label{keine-Ahnung}
 \parallel f\ast S_j\parallel_2^2=
\parallel \widehat{(f\ast S_j)}\parallel_{H.S}^2
  &=\int_{\RR^\ast}\parallel \widehat{(f\ast S_j)}(\lambda)\parallel_2^2  d\mu(\lambda) \\
\notag
 &=\int_{0<\left|4^{-j}\lambda\right|\leq \frac{\pi}{2d}}  \parallel \widehat f(\lambda)
 \circ \widehat{S_j}(\lambda)
 \parallel_{H.S}^2
 d\mu(\lambda)\\ \notag
&\leq\int_{0<\left|4^{-j}\lambda\right|\leq \frac{\pi}{2d}}
\parallel \widehat f(\lambda) \parallel_{H.S}^2 \parallel \widehat{S_j}(\lambda)\parallel_\infty^2
 d\mu(\lambda)\\\notag
&\leq \int_{0<\left|4^{-j}\lambda\right|\leq \frac{\pi}{2d}} \parallel \widehat f(\lambda)
\parallel_{H.S}^2
 d\mu(\lambda)\\\label{for-convergences}
&=\int_{\RR^\ast}\parallel \widehat f(\lambda) \parallel_{H.S}^2
\chi_{_{\mathcal{I}_j}}(\lambda) ~
d\mu(\lambda),
 \end{align}
 where $\chi$ denotes the characteristic function and $ d\mu(\lambda)=(2\pi)^{-2}
\left|\lambda\right|d\lambda$.
             If we take  the limit of the right hand side in (\ref{for-convergences}),
 since $\int_\lambda \parallel \widehat f(\lambda) \parallel_{H.S}^2
d\mu(\lambda)<\infty$,   then by the dominated convergence theorem
may pass the
limit into the integral and hence
 \begin{align}\notag
  \underset{j\rightarrow -\infty}{\lim}
 \int_{\RR^\ast}\parallel \widehat f(\lambda) \parallel_{H.S}^2
\chi_{_{\mathcal{I}_j}}~
d\mu(\lambda) 
 =0,
 \end{align}
 The latter   implies that the limit of the  left hand side in the  relation
  $(\ref{keine-Ahnung})$ is also zero as
 $j\rightarrow -\infty$, i.e.,
            $\underset{j\rightarrow -\infty}{\lim}\parallel f\ast S_j\parallel_2=0$,
        which proves $(b)$. \\
        In order to  prove $(c)$, suppose $f$ is in $L^2(\HH)$. Recall  that   
$\{e_i^{\frac{\lambda}{2\pi}} \}_{i=0}^\infty$
constitutes an orthonormal basis for $L^2(\RR)$ for any fixed $\lambda$. Therefore  the identity operator $I$ on $L^2(\RR)$ can
be read as
  $I=\sum_{i=0}^\infty e_i^{\frac{\lambda}{2\pi}}\otimes e_i^{\frac{\lambda}{2\pi}},$
and hence  the operator $\hat f(\lambda)$   can be represented as
\begin{align}\label{representation-for-f}
\widehat f( \lambda)=
\sum_{i=0}^\infty \left(\widehat f(\lambda) e_i^{\frac{\lambda}{2\pi}} \right)\otimes
e_i^{\frac{\lambda}{2\pi}}.
\end{align}
Therefore for any $j\in\ZZ$, according to the  representation of $\widehat f( \lambda)$
in (\ref{representation-for-f})  and  the  representation of the operator
 $ D_{2^{-j}}\widehat S(2^{-2j}\lambda)D_{2^j}$ in $(\ref{DSj})$, for some $k_j\geq j$
 we obtain  the followings:
   \begin{align}\notag
 \left\|  \left[\widehat f(\lambda)\circ  D_{2^{-j}}\widehat S(2^{-2j}\lambda)D_{2^j}\right]-
 \widehat f(\lambda)
\right\|_{H.S}^2  
    &=  \left\|\sum_{i=2^{2k_j}+1}^\infty \left(\widehat
f(\lambda)e_i^{\frac{\lambda}{2\pi}} \right)\otimes e_i^{\frac{\lambda}{2\pi}}
\right
\|_{H.S}^2\\\label{for-infty}
  &=  \sum_{i=2^{2k_j}+1}^\infty \left\| \widehat
f(\lambda)e_i^{\frac{\lambda}{2\pi}}\right\|_2^2.
 \end{align}
 Letting $j\rightarrow \infty$ (hence  $k_j\rightarrow \infty$),  the right hand side
 of $(\ref{for-infty})$ goes to zero. From the other side using the
            Plancherel theorem we have
           \begin{align}\label{just-an-equality}
 \parallel f\ast S_j -f\parallel_2^2&=\int_{\RR^\ast}
\left\|  \left[ \widehat f(\lambda)\circ  D_{2^{-j}}\widehat S(4^{-j}\lambda)D_{2^j}\right]-\widehat
f(\lambda)  \right\|_{H.S}^2 d\mu(\lambda)\\\notag
&= \int_{\RR^\ast}
\sum_{i=2^{2k_j}+1}^\infty \big\|\widehat
f(\lambda) e_i^{\frac{\lambda}{2\pi}} \big\|_2^2 d\mu(\lambda).
 \end{align}
As in the proof of 
    $(b)$, using  the dominated convergence theorem
   in the relation $(\ref{just-an-equality})$ one gets:
         \begin{align}\notag
 \underset{j\rightarrow \infty}{\lim}\parallel f\ast S_j -f\parallel_2^2  = \int_{\RR^\ast} \underset{j\rightarrow \infty}{\lim}\sum_{i=2^{2(k_j+j)}+1}^\infty \big\|\widehat
f(\lambda) e_i^{\frac{\lambda}{2\pi}} \big\|_2^2 d\mu(\lambda) =0,
 \end{align}
  as desired, which completes the proof of the theorem.
  \end{proof}
  
  \begin{remark} In the previous theorem, one 
could for instance  take the orthonormal basis of $\{\phi_n\}_{n\in \NN_0}$ in $L^2(\RR)$,
where
$\phi_n$ are  Hermite functions, and  for any $\lambda\not=0$, $\phi_n^\lambda$ are given  by
$\phi_n^\lambda(x)=D_{\left|\lambda\right|^{-1/2}}\phi_n=
  \left| \lambda\right|^{\frac{1}{4}}\phi_n(\sqrt{\mid \lambda\mid}x)$ for all $x\in \RR$.
\end{remark}
     
Now that we have constructed a function $S$ as in above theorem, with the listed properties,
the next step in the construction of an MRA via the function $S$  will be 
the
definition of a closed left invariant subspace of $L^2(\HH)$, $V_0$.
 Define $V_0= L^2(\HH)\ast S$, as the central subspace of an MRA. It is obvious that $V_0$ is  closed   and possesses 
 the following additional properties:
  \begin{enumerate}
    \item \;  $V_0$ is contained in the set of all  bounded and continuous functions in
$L^2(\HH)$. Hence
    $V_0$ is a proper subspace of  $L^2(\HH)$.
        The  boundedness of elements in $V_0$  is easy to see by the definition of convolution operator
and Cauchy-Schwartz inequality:
 $$
 \mid g\ast S(x)\mid \leq \parallel f\parallel_2 \parallel S\parallel_2 \quad \forall x\in\HH\quad
g\in L^2(\RR),
$$
\item\; Since $S$ is convolution idempotent then $S$ behaves as an identity element in $V_0$ with respect to group convolution. More precisely, 
  $f\ast S=f$ for any $f\in V_0$.
\item\; Suppose $\Gamma$ is  any lattice in $\HH$. Then
    $L_\gamma (g\ast S)= L_\gamma g\ast S$  which shows  $V_0$ is left shift-invariant under
$\Gamma$.
  \item\; An easy computation shows that 
   $D_{2^j}(g\ast S)= D_{2^j}g\ast D_{2^j}S$ for any $g\in L^2(\HH)$  and $j\in \ZZ$.
       \end{enumerate}

\begin{remark}
Observe  that not every space   $L^2(\HH)\ast S$ with
$S=\tilde{S}  =S\ast S$ gives rise to a normalized tight frame  of the form $\{L_\gamma \phi\}_\gamma$
 for some
$\phi\in L^2(\HH)\ast S$. As shall  be seen later, this depends heavily on the \textit{multiplicity
function}
 associated to $S$,
see Definition  \ref{multiplicity-function} and   Theorem \ref{fuehrstheorem}.
\end{remark}

 Recall that
$L_{2^{-j}\gamma}D_{2^{-j}}
 S(\omega)= 2^{j/2}S(\gamma^{-1}(2^{j}\omega)) \quad \forall j\in \ZZ\; ,\gamma\in \Gamma, \;
x\in
\HH.
$
 Next define $V_1= L^2(\HH)\ast (2^4 S(2\cdot))$.
$V_1$ is left-invariant under $2^{-1}\Gamma $ and is a closed subspace of
$L^2(\HH)$ as well. The functions in $V_1$  are  continuous bounded functions,   and from (\ref{S-j}) are
$\mathcal{I}_1$-band-limited.
  With regard to the consequence $(a)$ of Theorem
 \ref{sinc-type-function}, for any $f\in V_0$ we have
 $$f=f\ast S=f \ast (S\ast 2^4S(2\cdot))=
 (f\ast S)\ast (2^4S(2\cdot)).$$
 The latter shows that the  conclusion  $V_0\subseteq V_1$ holds. 
 By continuing in this manner, we define $V_2= L^2(\HH)\ast (2^8S(2^2\cdot))$ to be the
closed subspace
    of  functions  which are  $\mathcal{I}_2$-band-limited.
    Obviously,  with a similar argument as above, one can easily prove that  $V_1\subseteq V_2$.
\\
 Similarly, one can define subspaces $V_3\subseteq V_4\subseteq \cdots$. On the other hand one may
define negatively indexed subspaces.
 For example, we define
 $$V_{-1}= L^2(\HH)\ast 2^{-4}S(2^{-1}\cdot).$$
  This space contains the
 functions
 which are    $\mathcal{I}_{-1}$-band-limited  and obviously
   $V_{-1}\subseteq V_0$. Again, one may continue in this way to construct  the sequence of closed
and left  $(2^{-j}\Gamma)$-shift-invariant subspaces of $L^2(\HH)$:
 \begin{align}\label{nested-property}
 \{0\}\subseteq\cdots V_{-2} \subseteq V_{-1} \subseteq V_{0} \subseteq V_{1} \subseteq
 L^2(\HH),
\end{align}
which are scaled versions of the central space $V_0$.
Our next aim is to show that, in the sense of Definition \ref{def:MRA-H},  the sequence of  closed
subspaces $\{V_j\}$ forms a FMRA  of $L^2(\HH)$. For this reason we  must  show that the all
properties $1$-$6$  in Definition \ref{def:MRA-H} hold for  the sequence $\{V_j\}$.
But (\ref{nested-property}) proves  the nested property of $V_j's$.  
The density and trivial intersection of $V_j's$ are given in the next theorem: 
 
\begin{theorem}\label{about-MRA}
$\{V_j\}_{j\in \ZZ}$ is dense in $L^2(\HH)$ and has trivial intersection.
\end{theorem}
\begin{proof}
To show the density of  $\{V_j\}_{j\in \ZZ}$ in $L^2(\HH)$, i.e.,
$\overline{\bigcup_{j\in \ZZ} V_j}= L^2(\HH)$, suppose $P_j$ denotes the projection operator of
$L^2(\HH)$ onto
 $V_j$.
Then $P_j$ is given by
\begin{align}\label{definition-of-Projections}
P_j:\; f\rightarrow f\ast 2^{4j} S(2^{j}\cdot).
\end{align}
Therefore the density of  $\{V_j\}_{j\in \ZZ}$ in $L^2(\HH)$ is equivalent to saying that
 for any $f$ in $L^2(\HH)$   with
$P_j f=0 \;\forall j\in\ZZ$, we have  $f=0$.  Theorem
\ref{sinc-type-function} $(c)$ demonstrates this. More precisely
$$
0= P_jf=f\ast S_j\rightarrow f \;as\; j\rightarrow \infty
$$
which implies $f=0$.\\
 For the triviality of the intersection, observe that
 $f\ast 2^{4j}S(2^j\cdot)= f$ for any $f\in V_j$. Therefore for 
 any 
  $f\in \bigcap V_j$  we have  $f\ast 2^{4j}S(2^j\cdot)= f$ for all $j$ . Therefore $(b)$  in Theorem
 \ref{sinc-type-function} implies that $f=0$, as desired.
 \end{proof}
 The other significant properties of $V_j's$ are collected in the next remark: 
 \begin{remark}
 \begin{itemize}
     \item[(1)]  Property $4$ in Definition \ref{def:MRA-H} is trivial  from the construction of
$V_j$'s.  This property enables us to  pass up and down among the spaces $V_j$ by scaling
    \begin{align}\notag
    f\in V_j \Longleftrightarrow f(2^{k-j}\cdot) \in V_k.
    \end{align}
    \item[(2)] Generally, when $V_0$ is left shift-invariant under some lattice $\Gamma$, the
spaces $V_j$ are shift-invariant under $2^{-j}\Gamma$.
 We will return to this fact  later
and will show  how one can  choose an appropriate lattice $\Gamma$ such that
     it allows the construction of a wavelet frame on $\HH$.
      \item [(3)]   Observe that, by contrast to the  multiresolution analysis
on $\RR$, condition $6$  in Definition \ref{def:MRA-H} requires the existence of some frame
generator
 $\phi$, not necessarily $\phi=S$.
 This is due to the fact that we did not assume any other conditions for the selection of the
orthonormal basis $\{e_i^\lambda\}_i$  for the construction of the Hilbert-Schmidt operators
$\widehat S(\lambda)$ (respectively $S$).
   This is one difference between  our defined  MRA of
                      $L^2(\HH)$ and
the one defined for  $L^2(\RR)$. In the case of $\RR$ the sinc function by which the subspaces $V_j$'s are
defined, generates an ONB for $V_0$ and hence for all $V_j$,  under some other suitable discrete
subgroups of $\RR$. In our case on the Heisenberg group we shall show the existence of  a function
$\phi$ in $V_0$ such that its left translations under a suitable  $\Gamma$ form a normalized
tight
frame for $V_0$ and hence  for all $V_j$ under $2^{-j}\Gamma$.
\end{itemize}
 \end{remark}
 As we briefly mentioned  above,   we shall show the existence of a function
$\phi$ in $V_0$ such that property $6$ in Definition \ref{def:MRA-H} holds for $V_0$. We will
observe below that this fact strongly depends on the structure of $S$ and definition of
$V_0$.
 To achieve this  goal,  we recall the following definition here:
  \begin{definition}\label{multiplicity-function} Suppose $\mathcal{H}$ be a left-invariant
subspace
 of $L^2(\HH)$ and $P$ be
the projection operator of $L^2(\HH)$ onto $\mathcal{H}$. There exists a unique  associated
projection field
$(\widehat{ P}_{\lambda})_\lambda$  satisfying
$
 \widehat{P(f)}(\lambda)=\hat f(\lambda)\circ
 \widehat P_{\lambda}\quad \forall\;f\in L^2(\HH).
$
 The associated  multiplicity function $
{m}_\mathcal{H}$ is then  defined by
 \begin{align}\notag
  {m}_\mathcal{H}:\; \RR^\ast \rightarrow \NN_0\cup \{\infty\};\quad   {m}_\mathcal{H}(\lambda)=
rank(\widehat P_\lambda).
 \end{align}
 $\mathcal{H}$ is called \textit{band-limited} if the support of its associated multiplicity
function  $m_\mathcal{H}$, $\Sigma(\mathcal{H})$, is bounded in $\RR^\ast$.
    \end{definition}
     
     Following the notation of \cite{Fuehr05},
    the  next theorem provides a characterization of closed left  shift-invariant subspaces of
$L^2(\HH)$ which admit a tight frame.
However, before we state this theorem we need to introduce two numbers associated to a lattice
$\Gamma$. The number $d(\Gamma)$ refers to a positive integer number $d$ for which
$\alpha(\Gamma_d)= \Gamma$ for some $\alpha\in Aut(\HH)$, where  $\Gamma_d$ is a lattice in $\HH$
and is defined
by
\begin{align}\label{lattice}
\Gamma_d:=\big\{(m,dk,l+\frac{1}{2}dmk):\; m,k,l\in \ZZ\big\}.
\end{align}
  It is easy to check that $ \Gamma_d$ forms a group under the group operation
(\ref{group-action}). Observe that
    due to Theorem 6.2 in \cite{Fuehr05},
     such a strictly positive number $d$  exists
    and is uniquely determined.  As well, we define  $r(\Gamma)$ be the  unique positive real
satisfying
   \begin{align}\notag
  \Gamma\cap Z(\HH)=\left\{(0,0,r(\Gamma)k); \;k\in\ZZ\right\},
\end{align}
 where  $Z(\HH)$ denotes the center of $\HH$, $Z(\HH)= \{0\}\times \{0\}\times \RR\subset\HH$.
With the above notation  we state the following theorem.  
    \begin{theorem}\label{fuehrstheorem}$\cite{Fuehr05}$
  Suppose $\mathcal{H}$ is a left-invariant subspace of $L^2(\HH)$  and  $ {m}_\mathcal{H}$  is
its associated  multiplicity function. Then there exists a tight frame $($hence normalized tight
frame$)$ of the form $\{L_\gamma\phi\}_{\gamma\in \Gamma}$ with an appropriate $\phi\in \mathcal{H}$
if
and only if the inequality
    \begin{align}\label{inequality-for-multiplicity}
    m_\mathcal{H}(2\pi\lambda)\left| 2\pi\lambda\right|+
    m_\mathcal{H}\left(2\pi\lambda-\frac{1}{r(\Gamma)}\right)\left|
2\pi\lambda-\frac{1}{r(\Gamma)}\right|\leq \frac{1}{d(\Gamma)r(\Gamma)}    \end{align}
    holds for $m_\mathcal{H}$ almost everywhere.\\
     From  the  inequality
$(\ref{inequality-for-multiplicity})$ it can be read off that
   $\mathcal{H}$ is band-limited. In fact, the support of $m_\mathcal{H}$
    is contained in the interval
$\left[ -\frac{1}{d(\Gamma)r(\Gamma)}, \frac{1}{d(\Gamma)r(\Gamma)}\right]$ up to a set of measure
zero.  
    \end{theorem}
 
    Note that Theorem 6.4 in \cite{Fuehr05} refers to a different realization of t\-he
\-Sc\-hr\"odin\-ge\-r
    representations,
    hence we have the additional factor $2\pi$ in the relation
$(\ref{inequality-for-multiplicity})$.\\
  
    Theorem   \ref{lattice}  enables us to show the existence of a function $\phi$ in
$V_0$
which provides a tight frame for $V_0$. Therefore as a consequence we have  our next main result in this section:
    \begin{theorem}\label{existence_of_scaling_function}
    There exists a normalized tight frame   of the form
$\{L_\gamma\phi\}_{\gamma\in \Gamma}$ for an appropriate $\phi\in V_0$ and a suitable lattice
$\Gamma$ in
 $\HH$.
    \end{theorem}
  \begin{proof}
    For our purpose we pick  a lattice with $r(\Gamma)= \frac{1}{2\pi}$ and $d(\Gamma)= d$.
(Observe that it is possible due to Theorem 6.2 in \cite{Fuehr05}  to select a lattice with the
desired
associated numbers $r$ and $d$.)  From the definition of $V_0$,
$\{\widehat{S}(\lambda)\}_{\lambda\in\RR^\ast}$ is  the associated projection field of $V_0$
 with the multiplicity
function $m_{V_0}$ which is  given by
\begin{align}\notag
   m_{V_0}(2\pi\lambda)= rank(\widehat S({2\pi\lambda}))=
   \begin{cases}
         2^{2k}+1\;&
   \text{if} \;\;
     2\pi\lambda \in I_0^k\;\; \text{for some}~k\in\NN_0\\
0\;& \;\; \text{elsewhere}.
\end{cases}
\end{align}
     One can easily prove that the inequality in (\ref{inequality-for-multiplicity}) holds
for $m_{V_0}$. By the construction of $S$ in Theorem \ref{main-theorem-of-frames},
     $\hat S(\lambda)=0$ for any $\left|\lambda\right| > \frac{\pi}{2d}$  which implies:
     \begin{align}\notag
     \Sigma (m_{V_0})\subset
     \left[-\frac{\pi}{2d} \; , \; \frac{\pi}{2d}  \right]\subset
       \left[-\frac{2\pi}{d} \;,\; \frac{2\pi}{d}  \right]= \left[-\frac{1}{d(\Gamma)r(\Gamma)} \; ,\;
\frac{1}{d(\Gamma)r(\Gamma)}  \right].
     \end{align}
    Therefore all the conditions  of Theorem \ref{fuehrstheorem} hold for $V_0$. Hence there
exists
a function $\phi$, so-called scaling function,  such that
    for our selected lattice $\Gamma$, $L_{\Gamma}\phi$ forms a  normalized tight frame for $V_0$.
From this, property $6$ of Definition \ref{def:MRA-H} is satisfied.
\end{proof}
 
    \begin{corollary}\label{a-trivial-lemma}
   For any $j\in\ZZ$,   $\{L_{2^{-j}\gamma}D_{2^{-j}}\phi\}_\gamma$ constitutes  a normalized
tight frame
of  $V_j$.
    \end{corollary}

    As already mentioned, we have constructed our
 the MRA for $L^2(\HH)$, with the aim of finding an associated discrete
  wavelet
system in $L^2(\HH)$.
  More precisely,   we want to
 construct  a discrete wavelet system for $L^2(\HH)$ which is a normalized tight frame.  We will
study this in detail in the next section by  considering a ``scaling function''  $\phi$ in $V_0$.

\subsection{Existence of  Normalized Tight Wavelet Frame for the  Hei\-sen\-berg group}
\label{sec:ExistanceOfNormalizedShannonTightFrameForHeisenbergGroup}
 It is natural to try to obtain one normalized tight frame (n.t frame)  for $L^2(\HH)$ by
combining all the n.t frames $\{L_{2^{-j}\gamma}D_{2^{-j}}\phi\}_{\gamma\in \Gamma}$ of $V_j$'s. But
although $V_j\subseteq V_{j+1}$, the n.t frame for $V_j$ is not necessarily contained in the n.t frame
$\{L_{2^{-(j+1)}\gamma}D_{2^{-(j+1)}}\phi\}_{\gamma\in \Gamma}$ of $V_{j+1}$. Therefore 
the union
of all n.t frames for $V_j$'s does not necessarily constitute a n.t frame for $L^2(\HH)$.\\
To find an n.t frame for $L^2(\HH)$, we use the following standard approach. For every $j\in \ZZ$,
use $W_j$ to denote the orthogonal complement of $V_j$ in $V_{j+1}$, i.e., $V_{j+1}= V_j\oplus
W_j$, where the symbol $\oplus$ stands for orthogonal closed subspace. Suppose $Q_j$ denotes the
 orthogonal projection of $L^2(\HH)$ onto $W_j$. Then $P_{j+1}= P_j+Q_j$ and evidently:
\begin{align}\notag
V_j= \bigoplus_{k\leq j-1}W_k.
\end{align}

 The most important thing remaining unchanged is that, the spaces $W_j, \;j\in \ZZ$, retain
the scaling property from $V_j$. More precisely, 
 \begin{align}\label{scallingproperty}
 f\in W_j\Longleftrightarrow f(2^{k-j}.)\in W_k.
 \end{align}
Consequently we obtain the following    orthogonal decomposition:
\begin{align}\label{eq:directsum}
L^2(\HH)= \bigoplus_{j\in \ZZ }W_j.
\end{align}
From this decomposition of $L^2(\HH)$, it
  follows that each $f\in L^2(\HH)$ has a representation $f=\sum_j Q_jf$, where $Q_jf   \bot Q_kf$
for any  pair  of $j,k$, $j\not=k$.
 \vspace{.1in}

 Our goal is reduced to finding a n.t frame for $W_0$. If we can find such a n.t frame for
$W_0$, then by the scaling property  (\ref{scallingproperty}) and orthogonal decomposition of
$L^2(\HH)$ into the $W_j$'s in (\ref{eq:directsum}), we can easily get a n.t frame for space
$L^2(\HH)$. We explain this in detail, in the next Lemma:
 \begin{lemma}\label{for-whole}
 Suppose $\psi\in W_0$ and $\Gamma$ is a lattice in $\HH$ such that
 $\{L_\gamma\psi\}_{\gamma\in \Gamma}$ constitutes a n.t    frame of $W_0$.  Then the wavelet system
$\{L_{2^{-j}\gamma}D_{2^{-j}}\psi\}_{\gamma,j}$ is a n.t  frame of $L^2(\HH)$.
 \end{lemma}
  \begin{proof}
 Observe that this lemma is a consequence of the orthogonal decomposition of $L^2(\HH)$ into the $W_j$'s.
Suppose $f\in L^2(\HH)$. From (\ref{eq:directsum}), $f$ can be written as
$
 f=\sum_j Q_j(f).
$
 Therefore to  prove that the system
 $\{L_{2^{-j}\gamma}D_{2^{-j}}\psi\}_{\gamma,j}$ forms a n.t frame of
 $L^2(\HH)$, it is sufficient to show that for any $j$ the system
 $\{L_{2^{-j}\gamma}D_{2^{-j}}\psi\}_{\gamma}$ is a n.t frame of $W_j$. From the scaling property
of the spaces $W_j$ (\ref{scallingproperty}) we have $Q_j(f)(2^{-j}.)\in W_0$. Take $Q_j(f)=f_j$.
From the hypothesis of the Lemma one has
\begin{align}\notag
\|f_j(2^{-j}.)\|^2=\sum_{\gamma\in \Gamma} \left| \langle f_j(2^{-j}.), L_\gamma\psi\rangle\right|^2.
\end{align}
 Replacing $2^{2j}D_{2^j}f_j(\cdot)= f_j(2^{-j}\cdot)$ in the above  we obtain
  \begin{align}\label{some-equality}
 \| f_j\|^2= \|D_{2^j}f_j\|^2 =\sum_{\gamma\in \Gamma} \left|\langle D_{2^j}f_j, L_\gamma\psi\rangle\right|^2
 = \sum_{\gamma\in \Gamma} \left|\langle f_j, L_{2^{-j}\gamma }D_{2^{-j}}\psi\rangle\right|^2.
  \end{align}
  Summing over $j$ in (\ref{some-equality}) yields:
     \begin{align}\notag
  \|f\|^2= \underset{j\in \ZZ}{\sum}\|f_j\|^2=
  \underset{j,\gamma}{\sum} \left|\langle f_j, L_{2^{-j}\gamma
   }D_{2^{-j}}\psi\rangle\right|^2
   =\underset{j,\gamma}{\sum} \left|\langle f, L_{2^{-j}\gamma
   }D_{2^{-j}}\psi\rangle\right|^2,
   \end{align}
 as desired.
 \end{proof}
 
 By  Lemma \ref{for-whole} it remains to show
that the space $W_0$ contains a function $\psi$ generating a normalized tight frame of $W_0$.
  \begin{remark}
  By  the definition of  orthogonal projections $P_1, P_0$ in  $(\ref{definition-of-Projections}),$
for any
  $ f\in L^2(\HH)$
   we have
  \begin{align}\notag
  Q_0(f)= P_1(f)-P_0(f)
  =f\ast \left[(2^4S(2.))-S \right],
  \end{align}
which implies that 
\begin{align}\label{about-W}
  W_0= L^2(\HH)\ast \left[(2^4  S(2.) )-S \right].
  \end{align}
   Likewise for any $j$ one can see that
     \begin{align}\notag
     W_j&= L^2(\HH)\ast \left[(2^{4j}S(2^j.))-(2^{4(j-1)}S(2^{j-1}.))
\right]\quad \text{and},\\\notag
    Q_j(f)&= f\ast \left[(2^{4j}S(2^j.))-(2^{4(j-1)}S(2^{j-1})
\right]\quad \forall \;f\in L^2(\HH),
    \end{align}
    where $Q_j$, as earlier mentioned,  is the projection operator of $L^2(\HH)$ onto $W_j$.
\end{remark}
    The  representation of the space $W_0$ in (\ref{about-W}) suggests that we can get a n.t. frame for
$W_0$  by applying  Theorem \ref{fuehrstheorem}. We obtain this in the proof of the next theorem,
which is  the last main result of this work:

    \begin{theorem}\label{prove-of-wavelet}
    There exists a band-limited function $\psi\in L^2(\HH)$ and a lattice
$\Gamma$
in  $\HH$ such that
   the discrete wavelet system $\{L_{2^{-j}\gamma}D_{2^{-j}}\psi\}_{j,\gamma}$ forms a n.t. frame
of
$L^2(\HH)$.
    \end{theorem}
\begin{proof}
In order to proof the theorem, first we shall show that the space $W_0$ is band-limited and contains a function  such that its left translations under a suitable lattice $\Gamma$ forms a n.t. frame of $W_0$. Hence, the assertion of the theorem will follow  from Lemma \ref{for-whole} and Theorem \ref{prove-of-wavelet}. \\
   Due to the support of $S$, we have
  $
   \large\sum\left[(2^4S(2.))-S \right]\subset \left[-\frac{\pi}{d}\; ,\; \frac{\pi}{d}
\right],
$
   where $\Sigma$ stands for the support 
   on the
Fourier transform side, and is applied for 
    the  function  $(2^4S(2.))-S$.  Hence $W_0$ is band-limited. To prove  that the space $W_0$ contains a n.t. frame,
observe that
 by Corollary \ref{a-trivial-lemma} the set $\{ L_{2^{-1}\gamma}D_{2^{-1}}\phi\}_{\gamma\in \Gamma}$
is
 a n.t frame of $V_1$ for a suitable $\Gamma$.
 On the other hand, the projection of $V_1$ onto $W_0$, $Q_0$, is left invariant and hence
 for any $\gamma\in \Gamma$, we have $Q_0(L_{2^{-1}\gamma}D_{2^{-1}}\phi)=
   L_{2^{-1}\gamma}\big(Q_0(D_{2^{-1}}\phi)\big)$. Since the image of a n.t. frame  under a left
shift-invariant
   projection is again a n.t. frame of the image space, 
  the set   $\{Q_0\big(L_{2^{-1}\gamma}D_{2^{-1}}\phi\big)\}_\gamma=\{
  L_{2^{-1}\gamma}\big(Q_0(D_{2^{-1}}\phi\big)
  \}_\gamma$ constitutes  a n.t. frame for $W_0$, as desired.
  \end{proof}
    We conclude our  work with the following remark:
  \begin{remark}
 As mentioned earlier, in contrast to the case of  $\RR$, in the present work it is  not required that the
wavelet
function  $\psi$ contained in $W_0$ be constructed through the so-called scaling function $\phi$ in $V_0$.
   \end{remark}

{\bf Acknowledgment}~  The author is grateful to Hartmut F\"uhr for his helpful    
 discussions and also grateful to Daryl Geller and Hans G. Feichtinger for their useful comments.

     \end{document}